\documentclass[a4paper,11pt,twoside]{article}

\title{Discounted Penalty Function at Parisian Ruin for L\'evy Insurance Risk Process}

\author{R. Loeffen\footnote{ronnie.loeffen@manchester.ac.uk;
School of Mathematics, University of Manchester, United Kingdom}, \; Z. Palmowski\footnote{zbigniew.palmowski@gmail.com;
Faculty of Pure and Applied Mathematics,
Wroc\l aw University of Science and Technology,
Wyb. Wyspia\'nskiego 27, 50-370 Wroc\l aw, Poland, Corresponding author} \; and \; B.A. Surya\footnote{budhi.surya@msor.vuw.ac.nz; School of Mathematics and Statistics, Victoria University of Wellington, New Zealand}
}

\date{\today}

\usepackage{color,colordvi}
\usepackage{amsmath,amssymb,amsthm}
\usepackage{epsfig}
\usepackage{lscape}
\usepackage{rotating}
\usepackage{multicol}
\usepackage{hyperref}

\newtheorem{theorem}{Theorem}[section]

\newtheorem{lemma}[theorem]{Lemma}
\newtheorem{cor}[theorem]{Corollary}
\newtheorem{corol}[theorem]{Corollary}
\newtheorem{prop}[theorem]{Proposition}

\newcommand{\exit}{{\mbox{\, \vspace{3mm}}} \hfill\mbox{$\square$}}

\rm 

\numberwithin{equation}{section}

\begin{document}
\maketitle \pagestyle{myheadings} \markboth{R. Loeffen, Z. Palmowski and B.A. Surya} {Total Discounted Penalty Function at Parisian Ruin}
\begin{abstract}
In the setting of a L\'evy insurance risk process, we present some results regarding the Parisian ruin problem which concerns the occurrence of an excursion below zero of duration bigger than a given threshold $r$.  First, we give the joint Laplace transform of ruin-time and ruin-position (possibly killed at the first-passage time above a fixed level $b$), which generalises  known results concerning Parisian ruin. This identity can be used to compute the expected discounted penalty function via Laplace inversion.  Second, we obtain the $q$-potential measure of the process killed at Parisian ruin. The results have  semi-explicit expressions in terms of the $q$-scale function and the distribution of the L\'evy process.


\medskip


\end{abstract}

\section{Introduction}
Let $X=\{X_t:t\geq 0\}$ be a spectrally negative L\'evy process
defined on filtered probability space
$(\Omega,\mathcal{F},\{\mathcal{F}_t:t\geq 0\},\mathbb{P})$.  That
is to say that $X$ is a stochastic process starting from zero,
having stationary and independent increments with c\`adl\`ag sample
paths with no positive jumps.
To avoid degenerate cases we exclude
processes $X$ with monotone paths. As a strong Markov process,
we shall endow $X$ with probabilities
$\{\mathbb{P}_x,x\in\mathbb{R}\}$, such that
$\mathbb{P}_x(X_0=x)=1$. Furthermore, we denote by $\mathbb{E}_x$
expectation with respect to $\mathbb{P}_x$. We will use convention that
$\mathbb{P}=\mathbb{P}_0$ and $\mathbb{E}=\mathbb{E}_0$.

Our main object of interest in this paper is so-called Parisian ruin time $\tau_r$
representing the first time that the process $X$ has spent $r>0$
units of time consecutively below zero before getting back up to
zero again. Formally, this stopping time is defined by
\begin{equation}\label{eq:ruintime}
\tau_r:=\inf\{t>r:(t-g_t)\geq r\} \; \; \textrm{with} \;\;
g_t:=\sup\{0\leq s \leq t: X_s\geq 0 \},
\end{equation}
under $\mathbb{P}_x$, with the convention that $\inf
\emptyset=\infty$ and $\sup \emptyset =0$.
Note that $\tau_0=\tau_0^-$, where
\begin{equation*}
\tau_0^-:=\inf\{t\geq 0: X_t<0\}
\end{equation*}
represents the classical ruin time.

The stopping time
$\tau_r$ defined in (\ref{eq:ruintime}) was first introduced by Chesney et al.
\cite{Chesney} in the context of pricing barrier options in
mathematical finance. It was later introduced in actuarial risk
theory by Dassios and Wu \cite{Dassios2009a} and they provided an expression for the Parisian ruin
probability $\mathbb{P}_x(\tau_r<\infty)$ when the underlying process is a linear Brownian motion. Czarna and Palmowski
\cite{Czarna} and Loeffen et al. \cite{Loeffen} extended the result
to a general spectrally negative L\'evy process.
Dassios and Wu \cite{Dassios2009b} provided the Laplace transform of the Parisian ruin time $\mathbb{E}_x \left[ \mathrm e^{-q\tau_r } \mathbf{1}_{\{\tau_r<\infty\}} \right]$ for the classical risk process with exponentially distributed claim sizes and the case of linear Brownian motion.
For a spectrally negative L\'evy processes of bounded variation  Landriault et al. \cite{Landriault} also considered the Laplace transform of the Parisian ruin time (possibly killed when the process goes above a given level) but in the setting where the delay is not a fixed number but is random with a mixed Erlang distribution which is resampled every time a new excursion below 0 starts.
In the same setting but with exponentially distributed delays, Baurdoux et al. \cite{Baurdoux} derived an expression for the (discounted) distribution of the process $X$ at the time of Parisian ruin possibly killed when exiting a given interval. Note that this particular definition of Parisian ruin is equivalent (i.e. the two ruin times have the same distribution) to so-called Poisson ruin which corresponds to $X$ dropping below zero in the setting where the L\'evy process is only monitored at the jump times of an independent Poisson process, see e.g. \cite{AIZ}.
For a refracted L\'evy process and with a fixed delay for the Parisian ruin time, Lkabous et al. \cite{CzarnaRenaud} determined the Laplace transform of the Parisian ruin time $\tau_r$ possibly killed when the process goes above a given level.

In this paper we build further upon
the previous works in the following directions. Firstly, we extend the works of \cite{Czarna, Dassios2009a, Dassios2009b, Loeffen} by giving
the joint Laplace transform of the Parisian ruin time $\tau_r$ and the level of the process at $\tau_r$ for a general spectrally negative L\'evy process. Secondly, using the first result we identify the $q$-potential measure applied to exponential functions of the spectrally negative L\'evy process killed at the Parisian ruin time.
The two results can be used to compute, via Laplace inversion, the following object:
\begin{equation}\label{eq:value}
V_r^{(q)}(x,b):=\mathbb{E}_x \Big[ \int_0^{\tau_r\wedge \tau_b^+}\mathrm e^{-qt} g(X_t)dt \Big]
+ \mathbb{E}_x\big[\mathrm e^{-q(\tau_r\wedge \tau_b^+)}f(X_{\tau_r\wedge \tau_b^+}) \big],
\end{equation}
where $g$ and $f$ are payoff functions and where
\[\tau_b^+:=\inf\{t\geq 0: X_t> b\}\]
with $b\in[0,\infty]$ and $x\leq b$. In the actuarial risk theory literature the last term on the right hand side is referred to as an expected discounted penalty function.
The expression in (\ref{eq:value}) is widely used in financial modelling,
in particular in the field of optimal capital structure under bankruptcy preceding and reorganization of a firm.  Following Broadie et al. \cite{Broadie} and Francois and Morellec \cite{Francois}, the first term in (\ref{eq:value}) may be interpreted as the total discounted payoff received prior to ruin of a financial firm payable as long as the firm's asset process stays above a certain pre-determined level, whereas the second term is the cost at ruin. We refer to the expression on page 393 in \cite{Francois}. According to \cite{Broadie} and \cite{Francois}, the stopping time $\tau_r$ is called the {liquidation time}. We refer
to \cite{Broadie} and \cite{Francois} and the literature therein for further details.

Apart of providing identities that allow one to compute the above quantity \eqref{eq:value}, another contribution of this paper lies in presenting some new methodology of calculating Parisian-type quantities. In particular for the key lemma (Lemma \ref{lem_key} below) we apply the method for dealing with overshoots of spectrally negative L\'evy processes presented in Loeffen \cite{ronnie} and the Kolmogorov forward equations, which avoids the need for taking Laplace transforms with respect to the delay $r$ and then inverting back later. Besides being able to compute more general quantities, this more direct approach also provides more transparent proofs of known results like the Parisian ruin probability obtained in \cite{Loeffen} in which the technique of taking Laplace transforms with respect to $r$ was heavily used.


The paper is organized as follows. In Section \ref{sec:prel} we recall some well-known results on spectrally negative L\'evy processes. Then the main results are presented in Section \ref{sec:mainresults}, whereas Section \ref{sec:proofs} contains the proofs.

\section{Preliminaries}\label{sec:prel}

For the spectrally negative L\'evy process $X$, there exists $\mu\in\mathbb{R}$, $\sigma\geq0$  and a measure $\Pi$ satisfying $\int_{-\infty}^0 (1\wedge \theta^2)\Pi(\mathrm{d}\theta)<\infty$ such that the Laplace transform of $X_t$  is given by, for any $x\in\mathbb R$ and $\lambda,t\geq 0$,
\begin{equation}\label{lapl_X}
\mathbb E_x \left[ \mathrm e^{\lambda X_t} \right] = \mathrm e^{\lambda x + \psi(\lambda)t},
\end{equation}
where
\begin{equation}\label{eq:exponent}
\psi(\lambda) 
=\mu\lambda + \frac{1}{2}\sigma^{2}\lambda^{2}  + \int_{(-\infty,0)} \left( \mathrm e^{\lambda\theta}-1-\lambda \theta\mathbf{1}_{\{\theta>-1\}} \right) \Pi(\mathrm{d}\theta).
\end{equation}
%
%
It is easily shown that $\psi$ is zero at the origin, tends to infinity at infinity and is strictly
convex. We denote by $\Phi:[0,\infty)\rightarrow [0,\infty)$ the
right continuous inverse of $\psi$ so that it satisfies the following:
\begin{equation*}
\Phi(q)=\sup\{\lambda\geq 0:\psi(\lambda)=q\}. 
\end{equation*}
Note that due to the convexity of $\psi$, there exit at most two
roots for a given $q$ and precisely one root when $q>0$.
Our main results are expressed in terms of the $q-$scale
function $W^{(q)}(x)$ of $X$, which satisfies $W^{(q)}(x)=0$ for $x<0$ and on $[0,\infty)$, $W^{(q)}(x)$ is the (unique) continuous function with Laplace transform,
\begin{equation}\label{eq:scale}
\int_0^{\infty} \mathrm e^{-\lambda x} W^{(q)}(x)\mathrm{d}x =
\frac{1}{\psi(\lambda)-q},  \quad   \lambda > \Phi(q).
\end{equation}
Following (\ref{eq:scale}), it is
straightforward to check that for all $x\in\mathbb R$ and $q\geq 0$,
\begin{equation}\label{eq:scale2}
W^{(q)}(x) =\mathrm e^{\Phi(q) x} W^{(0)}_{\Phi(q)}(x),
\end{equation}
where $W^{(0)}_{\Phi(q)}(x)$ is the $0$-scale function of the spectrally negative L\'evy process with Laplace exponent $\lambda\mapsto\psi(\lambda+\Phi(q))-q$.
We remark that the scale function is a strictly increasing function, that $W^{(q)}(0)>0$ if  $\sigma=0$ and $\int_{-1}^0\theta\Pi(\mathrm d\theta)<\infty$ and otherwise $W^{(q)}(0)=0$. Further, when $q>0$, we have $\lim_{x\to\infty}W_{\Phi(q)}(x)=\frac1{\psi'(\Phi(q))}<\infty$.



For further details on spectrally negative L\'evy process, we refer
to Chapter VI of Bertoin \cite{Bertoin} and Chapter 8 of Kyprianou
\cite{Kyprianou}. Some examples of L\'evy processes for $W^{(q)}(x)$
are available in explicit form can be found in Kuznetzov et al.
\cite{Kuznetzov}. In any case, it can be computed by numerically
inverting (\ref{eq:scale}), see Surya \cite{Surya}.


\section{Main results}\label{sec:mainresults}

In the main results the function $\Lambda^{(q)}(x,r)$ defined by
\begin{equation}\label{eq:Lambdaq}
\Lambda^{(q)}(x,r)=\int_0^{\infty}W^{(q)}(x+z)\frac{z}{r}\mathbb{P}(X_r\in \mathrm{d}z)
\end{equation}
will frequently appear. By Kendall's identity $\mathbb P(\tau_z^+\in\mathrm d r)\mathrm{d}z =   \frac{z}{r}\mathbb{P} (X_r\in\mathrm{d}z)\mathrm d r$,  $z,r\geq 0$, (cf. Corollary VII.3 in \cite{Bertoin}) and the fact $\mathbb E \left[ \mathrm e^{-\lambda \tau_z^+} \right] = \mathrm e^{-\Phi(\lambda)z}$, $\lambda,z\geq 0$, (cf. Section 8.1 in \cite{Kyprianou}) it follows that the Laplace transform of $r\mapsto\mathrm e^{-q r}\Lambda^{(q)}(x,r)$ is given by
\begin{equation*}
\int_0^\infty  \mathrm e^{-\theta r} \left( \mathrm e^{-q r} \Lambda^{(q)}(x,r) \right) \mathrm d r =  \int_0^\infty  \mathrm e^{-\Phi(\theta+q) z} W^{(q)}(x+z)\mathrm d z, \quad \theta>0, x\in\mathbb R.
\end{equation*}
It is interesting to note that in a similar role as $ \Lambda^{(q)}(x,r)$ will appear in our identities, the right hand side of the above equation appears in identities involving Parisian ruin with exponentially distributed delays with parameter $\theta$, or equivalently, ruin when the L\'evy process $X$ is observed only at the jump times of an independent Poisson process with rate $\theta$, see e.g. Proposition 2.1 in \cite{Landriault}, the function $\mathcal H^{(q+\theta,-\theta)}(x)$ appearing in \cite{Baurdoux} and the function $Z_q(x,\Phi(\theta+q))$ appearing in \cite{AIZ}. This connection is somewhat surprising since although taking Laplace transforms in $r$ is equivalent to considering an exponentially distributed delay (sampled once (and independently)), this is different from the setup in the papers \cite{Baurdoux} and \cite{Landriault} where the exponentially distributed delay is resampled for each excursion below zero.

The first main results concerns the joint Laplace transform of the Parisian ruin time and overshoot killed at the first-passage
time above a fixed level $b$.

\begin{theorem}\label{thm_lapl}
For $q,\lambda\geq 0$, $r,b>0$ and $x\leq b$,
\begin{equation*}
\begin{split}
  \mathbb{E}_x & \left[ \mathrm{e}^{-q ( \tau_r - r )} \mathrm e^{ \lambda X_{\tau_r} -\psi(\lambda)r } \mathbf{1}_{\{\tau_r<\tau_b^+\}} \right] \\
= &   \mathrm e^{\lambda x } -  (\psi(\lambda)-q)  \left[ \int_0^x W^{(q)}(x-z)  \mathrm e^{\lambda z}  \mathrm d z
 + \int_0^r   \mathrm e^{-\psi(\lambda) s } \Lambda^{(q)}(x,s)   \mathrm{d}s \right] \\
& - \frac{ \Lambda^{(q)}(x,r)   }{ \Lambda^{(q)}(b,r)   } \left(  \mathrm e^{\lambda b }  -  (\psi(\lambda)-q) \left[  \int_0^b W^{(q)}(b-z)  \mathrm e^{\lambda z}  \mathrm d z + \int_0^r   \mathrm e^{-\psi(\lambda) s } \Lambda^{(q)}(b,s)   \mathrm{d}s  \right] \right).
\end{split}
\end{equation*}
\end{theorem}

For the special case where $\lambda=0$, Lkabous et al. \cite[Thm. 4(i)]{CzarnaRenaud} (set $\delta=0$ there) provide a similar expression for the left hand side of the above identity. Their expression is slightly different than ours, which is because in the proof they use the first identity in Lemma \ref{lem_key} below (in the special case where $f\equiv 1$ and $\widetilde f\equiv 1$) with $p=0$ whereas we choose $p=q$, which leads to nicer expressions in the case where $\lambda>0$.


By setting $q=\lambda=0$ in Theorem \ref{thm_lapl} and performing an Esscher change of measure, one can get the next corollary. It is a special case of \cite[Thm. 4(iii)]{CzarnaRenaud} (again by taking $\delta=0$ there) and we refer to \cite{CzarnaRenaud} for further details of the proof.
\begin{cor}\label{LTpartwosided}
For $q\geq 0$, $r,b>0$ and $x\leq b$,
\begin{equation*}
\mathbb{E}_x \left[ \mathrm e^{-q\tau_b^+}\mathbf{1}_{\{\tau_b^+<\tau_r \}} \right] =\frac{\Lambda^{(q)}(x,r)}{\Lambda^{(q)}(b,r)}.
\end{equation*}
\end{cor}


Next we give an expression for the Laplace transform of the $q$-potential measure killed at Parisian ruin or when $X$ goes above $b$, whatever comes first.

\begin{theorem}\label{theo:main2}
For $q,\lambda\geq 0$, $r,b>0$ and $x\leq b$,
\begin{equation*}
\begin{split}
& \mathbb{E}_x  \left[\int_0^{\tau_r\wedge \tau_b^+} \mathrm e^{-q(t-r)} \mathrm e^{\lambda X_t -\psi(\lambda)r} \mathrm{d}t \right] \\
= &   \frac{ \mathrm e^{\lambda x} \left( 1 - \mathrm e^{-(\psi(\lambda)-q) r} \right)  }{ \psi(\lambda)-q} -     \int_0^x W^{(q)}(x-z)  \mathrm e^{\lambda z}  \mathrm d z   - \int_0^r   \mathrm e^{-\psi(\lambda)s } \Lambda^{(q)}(x,s) \mathrm{d}s  \\
 - & \frac{ \Lambda^{(q)}(x,r)   }{ \Lambda^{(q)}(b,r)   } \Bigg(  \frac{ \mathrm e^{\lambda b} \left( 1 - \mathrm e^{-(\psi(\lambda)-q) r} \right) }{ \psi(\lambda)-q}   -      \int_0^b W^{(q)}(b-z)  \mathrm e^{\lambda z}  \mathrm d z  - \int_0^r   \mathrm e^{-\psi(\lambda)s } \Lambda^{(q)}(b,s)   \mathrm{d}s \Bigg).
\end{split}
\end{equation*}
%
\end{theorem}

We can actually invert this Laplace transform to get an expression for this $q$-potential measure itself, see Theorem \ref{theo:general2} below. This expression is quite difficult to numerically evaluate and therefore we recommend to   invert numerically the Laplace transform in Theorem \ref{theo:main2} for computing the aforementioned $q$-potential measure.
However, it turns out that there is a very simple expression for this $q$-potential measure when restricted on the positive half-line and we provide this result here.
\begin{prop}\label{prop_resolpos}
	Let $r,b>0$ and $q\geq 0$. Then for $x\leq b$ and $y\geq 0$,
	\begin{equation*}
	\begin{split}
	\int_0^\infty \mathrm e^{-q t}  \mathbb P_x(X_t\in\mathrm d y&,  t<\tau_r\wedge\tau_b^+) \mathrm d t = \left( \frac{\Lambda^{(q)}(x,r)}{\Lambda^{(q)}(b,r)} W^{(q)}(b-y) - W^{(q)}(x-y)  \right) \mathrm d y.
	\end{split}
	\end{equation*}
	
\end{prop}

Finally, we provide the versions of Theorems \ref{thm_lapl} and \ref{theo:main2} when $b=\infty$.
\begin{corol}\label{cor_lapl_limb}
	For $q,\lambda\geq 0$, $r>0$ and $x\in\mathbb R$,
	\begin{equation*}
	\begin{split}
	\mathbb{E}_x & \left[ \mathrm{e}^{-q ( \tau_r - r )} \mathrm e^{ \lambda X_{\tau_r} -\psi(\lambda)r } \mathbf{1}_{\{\tau_r<\infty\}} \right] \\
	= &   \mathrm e^{\lambda x } -  (\psi(\lambda)-q)  \left[ \int_0^x W^{(q)}(x-z)  \mathrm e^{\lambda z}  \mathrm d z
	+ \int_0^r   \mathrm e^{-\psi(\lambda) s } \Lambda^{(q)}(x,s)   \mathrm{d}s \right] \\
	& - \frac{ \Lambda^{(q)}(x,r) (\psi(\lambda)-q)  }{ \int_0^\infty \mathrm e^{\Phi(q)z}\frac zr\mathbb P(X_r\in\mathrm d z) } \left(  \frac{1}{\lambda-\Phi(q)} +   \int_0^r   \mathrm e^{-\psi(\lambda) s } \int_0^\infty \mathrm e^{\Phi(q)z}\frac zs\mathbb P(X_s\in\mathrm d z)   \mathrm{d}s \right),
	\end{split}
	\end{equation*}
where for $\lambda=\Phi(q)$ the ratio $\frac{(\psi(\lambda)-q)}{\lambda-\Phi(q)}$ is understood in the limiting sense $\lim_{\lambda\to\Phi(q)}\frac{(\psi(\lambda)-q)}{\lambda-\Phi(q)}=\psi'(\Phi(q))$.
\end{corol}

\begin{corol}\label{cor_laplresol_limb}
	Let $q\geq 0$, $r>0$ and $x\in\mathbb R$. Then, if  $\lambda\in[0,\Phi(q))$,
	\begin{equation*}
	\begin{split}
	\mathbb{E}_x & \left[\int_0^{\tau_r} \mathrm e^{-q(t-r)} \mathrm e^{\lambda X_t -\psi(\lambda)r} \mathrm{d}t \right] \\
	= &    \frac{ \mathrm e^{\lambda x} \left( 1 - \mathrm e^{-(\psi(\lambda)-q) r} \right)  }{ \psi(\lambda)-q} -      \int_0^x W^{(q)}(x-z)  \mathrm e^{\lambda z}  \mathrm d z
	- \int_0^r   \mathrm e^{-\psi(\lambda) s } \Lambda^{(q)}(x,s)   \mathrm{d}s  \\
	& - \frac{ \Lambda^{(q)}(x,r)  }{ \int_0^\infty \mathrm e^{\Phi(q)z}\frac zr\mathbb P(X_r\in\mathrm d z) } \left(  \frac{1}{\lambda-\Phi(q)} +   \int_0^r   \mathrm e^{-\psi(\lambda) s } \int_0^\infty \mathrm e^{\Phi(q)z}\frac zs\mathbb P(X_s\in\mathrm d z)   \mathrm{d}s \right),
	\end{split}
	\end{equation*}
	whereas if $\lambda\geq\Phi(q)$, then the left hand side equals $+\infty$.
\end{corol}
Note that by setting $q=\lambda=0$ in Corollary \ref{cor_lapl_limb}, we recover the expression for the Parisian ruin probability found in \cite{Loeffen}.

\section{Proofs of main results}\label{sec:proofs}

\subsection{Proof of Theorem \ref{thm_lapl}}

We first introduce some further notations. Given the spectrally negative L\'evy process with Laplace exponent given by \eqref{eq:exponent}, we let $\mathcal A$ be the operator
\begin{equation}
\label{generator}
\mathcal{A}h(x)=\mu h^{\prime}(x) + \frac{\sigma}{2} h''(x)+
\int_{-\infty}^{0} (h(x+\theta)-h(x) -\mathbf{1}_{\{\theta>-1\}}
y h^{\prime}(x))\Pi(\mathrm{d}\theta),
\end{equation}
where $h:\mathbb R\to\mathbb R$ is any function such that the right hand side above is well-defined. Note that $\mathcal A$ coincides with the infinitesimal generator of $X$  on $C^2_0(\mathbb R)$, the space of twice continuously differentiable functions that together with its first two derivatives vanish at $+\infty$ and $-\infty$, see e.g. \cite[Theorem 31.5]{Sato}.

The assertion of Theorem \ref{thm_lapl} is a simple consequence of the following more general result, which provides the expected discounted penalty function at Parisian ruin, provided this happens before $\tau_b^+$. This result can be seen as an extension of Corollary 3 in \cite{ronnie} to the case where $r>0$. There are several conditions imposed on the penalty function $f$ and its extension $\widetilde f$.
We remark that conditions (i)-(vii) below are all satisfied when $\widetilde f$ is in $C^2_0(\mathbb R)$.

\begin{theorem}\label{thm_mainresult}
Let $r,b>0$ and $q\geq 0$.  
Let $f:(-\infty,0]\to\mathbb R$ be a function such that there exists $\widetilde f:\mathbb R\to\mathbb R$ with $\widetilde f|_{(-\infty,0]}\equiv f$ and satisfying the following conditions:
\begin{itemize}
	\item[(i)] $\widetilde f$ is twice continuously differentiable on $\mathbb R$ and for any $x\in\mathbb R$, the following Kolmogorov forward equation holds:
\begin{equation*}
 \mathrm e^{-q r} \mathbb E_x \left[ \widetilde f(X_r) \right] = \widetilde f(x) +  \int_0^r    \mathrm e^{-q s}  \mathbb E_x \left[ (\mathcal A-q)\widetilde f (X_{s})  \right] \mathrm d s;
\end{equation*}
	
	\item[(ii)]  for any $x\in\mathbb R$, $\sup_{s\in[0,r]} \mathbb E_x \left[ \left| (\mathcal A-q)\widetilde f (X_{s}) \right| \right] <\infty$;
	\item[(iii)] for some $\delta>0$, $x\mapsto \mathbb E_x [\widetilde f(X_r)]$ is twice continuously differentiable on $(-\delta,b)$;
	\item[(iv)] for some $z>0$, $\sup_{x\in[0,b]}  \int_{-\infty}^{-z} \left| \mathbb E_{x+\theta} \left[ \widetilde f(X_r) \right] \right| \Pi(\mathrm d \theta)<\infty$;
	
	\item[(v)] for some $\delta\in(0,b)$,  $\sup_{x\in[0,\delta]} \left| (\mathcal A-q) \mathbb E_x [\widetilde f(X_r)]\right|<\infty$;
	
	\item[(vi)] for some $\delta>0$, $x\mapsto\mathbb E_x \left[  (\mathcal A-q)\widetilde f (X_{s})  \right]$ is continuous on $[0,\delta]$ for all $s\in[0,r]$;
	
	\item[(vii)] $\sup_{x\in(-\infty,b]}|\widetilde f(x)|<\infty$.
\end{itemize}
Then for $x\leq b$,
\begin{equation*}
\begin{split}
\mathbb{E}_x   \left[ \mathrm{e}^{-q(\tau_r-r)} f(X_{\tau_r})\mathbf{1}_{\{\tau_r<\tau_b^+\}} \right]
= &  \zeta(x;r,q) - \frac{ \Lambda^{(q)}(x,r)   }{ \Lambda^{(q)}(b,r)   }\zeta(b;r,q),
\end{split}
\end{equation*}
where for $y\leq b$,
\begin{equation*}
\begin{split}
\zeta(y;r,q)
:= & \mathbb E_{y} \left[ \widetilde f (X_r)  \right] -  \int_0^y W^{(q)}(y-z)    (\mathcal A-q) \mathbb E_{z} \left[ \widetilde f (X_r)  \right]  \mathrm d z\\
& -  \int_0^r  \mathbb E \left[ (\mathcal A-q)\widetilde f (X_{r-s})  \right]  \Lambda^{(q)}(y,s)   \mathrm{d}s.
\end{split}
\end{equation*}
\end{theorem}

In the majority of cases for the function $f$ it should be easier to compute this expected discounted penalty function via Theorem \ref{thm_lapl} and Laplace inversion rather than using Theorem \ref{thm_mainresult} directly. The reason why we still include Theorem 4.1 for a class of functions $f$ satisfying (i)-(vii) rather than working directly with $f(x)=\mathrm e^{\lambda x}$, is because in this way the proof of Theorem \ref{thm_lapl} becomes much more transparent.




\paragraph{Proof of Theorem \ref{thm_lapl}}
Let $\lambda\geq 0$. We want to use Theorem \ref{thm_mainresult} with $f(x)=\mathrm e^{\lambda x}$. We choose $\widetilde f(x)=\mathrm e^{\lambda x}$ for $x\in\mathbb R$. It is easy to show via \eqref{lapl_X}, \eqref{eq:exponent} and \eqref{generator} that then, for $x\in\mathbb R$,
\begin{equation*}
\mathbb E_x[\widetilde f(X_r)] = \mathrm e^{\lambda x+\psi(\lambda)r}, \quad (\mathcal A-q) \widetilde f(x) = \mathrm e^{\lambda x}(\psi(\lambda)-q).
\end{equation*}
Hence all the conditions on $\widetilde f$ in Theorem \ref{thm_mainresult} are satisfied and Theorem \ref{thm_lapl} follows. \exit


\subsection{Proof of Theorem \ref{thm_mainresult}}

We start by providing two lemmas. The first one follows from \cite[Lem. 8]{CzarnaRenaud} (where $\delta=0$ in \cite{CzarnaRenaud}) and the spatial homogeneity of $X$.
\begin{lemma}\label{lem_overshoot_Lamq}
Let $r,b>0$, $q\geq 0$ and let $\varepsilon\in[0,b)$. Then for $x\leq \varepsilon$,
\begin{equation*}
\mathbb E_x \left[ \mathrm e^{-q\tau_\varepsilon^+} \mathbf 1_{\{\tau_\varepsilon^+\leq r\}}  \right] =   \mathrm e^{-q r} \Lambda^{(q)}(x-\varepsilon,r)
\end{equation*}
and for $x\leq b$,
\begin{equation*}
\begin{split}
\mathbb{E}_x \Big[ \mathrm{e}^{-q(\tau_0^- -r)} \mathbf{1}_{\{\tau_0^-<\tau_b^+\}}
\mathbb{E}_{X_{\tau_0^-}} \left[ \mathrm e^{-q\tau_\varepsilon^+} \mathbf 1_{\{\tau_\varepsilon^+\leq r\}} \right] \Big]
= &  \mathbb{E}_x \left[ \mathrm{e}^{-q\tau_0^-} \mathbf{1}_{\{\tau_0^-<\tau_b^+\}} \Lambda^{(q)}(X_{\tau_0^-}-\varepsilon,r) \right] \\
= &   \Lambda^{(q)}(x-\varepsilon,r) - \frac{W^{(q)}(x)}{W^{(q)}(b)}  \Lambda^{(q)}(b-\varepsilon,r).
\end{split}
\end{equation*}
\end{lemma}

%
%

The second lemma is the key and fully original part of the proof, which allows us to deal with the overshoot at Parisian ruin.
\begin{lemma}\label{lem_key}
Let $r,b>0$, $q\geq 0$ and assume $f$ and $\widetilde f$ are as in Theorem \ref{thm_mainresult}.
Let $\varepsilon\in[0,b)$. Then for $x\leq \varepsilon$ and $p\geq 0$,
\begin{multline*}
\mathbb{E}_{x} \left[ \widetilde{f}(X_r) \mathbf 1_{\{r<\tau_\varepsilon^+\}} \right]
=    \mathbb E_{x} \left[ \widetilde f (X_r)  \right] -  \widetilde f(\varepsilon) \Lambda^{(p)}(x-\varepsilon,r)  \\
  - \int_0^r   \mathbb E_\varepsilon \left[ (\mathcal A-p)\widetilde f (X_{r-s})  \right] \Lambda^{(p)}(x-\varepsilon,s) \mathrm d s
\end{multline*}
and for $x\leq b$,
\begin{multline*}
\mathbb{E}_x \left[ \mathrm{e}^{-q\tau_0^-} \mathbf{1}_{\{\tau_0^-<\tau_b^+\}}
\mathbb{E}_{X_{\tau_0^-}} \left[ \widetilde{f}(X_r) \mathbf 1_{\{r<\tau_\varepsilon^+\}} \right] \right] \\
= \zeta_\varepsilon(x;r,q) - \widetilde f(\varepsilon)  \Lambda^{(q)}(x-\varepsilon,r) - \frac{W^{(q)}(x)}{W^{(q)}(b)}  \left( \zeta_\varepsilon(b;r,q)- \widetilde f(\varepsilon)  \Lambda^{(q)}(b-\varepsilon,r) \right),
\end{multline*}
where 
\begin{equation*}
\begin{split}
\zeta_\varepsilon(y;r,q)
:= & \mathbb E_{y} \left[ \widetilde f (X_r)  \right] -  \int_0^y W^{(q)}(y-z)    (\mathcal A-q) \mathbb E_{z} \left[ \widetilde f (X_r)  \right]  \mathrm d z \\
& -  \int_0^r  \mathbb E_\varepsilon \left[ (\mathcal A-q)\widetilde f (X_{r-s})  \right]  \Lambda^{(q)}(y-\varepsilon,s)   \mathrm{d}s.
\end{split}
\end{equation*}
\end{lemma}
\begin{proof}
We start with proving the first identity. Let $x\leq \varepsilon$. Then,
	\begin{equation*}
	\begin{split}
	\mathbb{E}_{x} \left[ \widetilde{f}(X_r) \mathbf 1_{\{r<\tau_\varepsilon^+\}} \right] 
	= & \mathbb E_{x} \left[ \widetilde f (X_r)  \right] - \mathbb{E}_{x} \left[ \widetilde f(X_r) \mathbf 1_{\{\tau_\varepsilon^+\leq r\}} \right]  \\
	= & \mathbb E_{x} \left[ \widetilde f (X_r)  \right] - \int_0^\infty \mathbb{E}_{x} \left[ \left. \widetilde f(X_r) \mathbf 1_{\{\tau_\varepsilon^+\leq r\}}  \right| \tau_\varepsilon^+=s \right]  \mathbb P_x(\tau_\varepsilon^+\in \mathrm d s) \\
	= & \mathbb E_{x} \left[ \widetilde f (X_r)  \right] - \int_0^r \mathbb{E}_{x} \left[ \left. \widetilde f(X_r-X_s+\varepsilon)   \right| \tau_\varepsilon^+=s \right]  \mathbb P_x(\tau_\varepsilon^+\in \mathrm d s) \\
	= & \mathbb E_{x} \left[ \widetilde f (X_r)  \right] - \int_0^r  \mathbb E_\varepsilon \left[ \widetilde f (X_{r-s})  \right]\mathbb P_x(\tau_\varepsilon^+\in\mathrm d s),
	\end{split}
	\end{equation*}
	where the last equality is due to the stationarity and independence of the increments of $X$ and the equality before that is because $X_s=\varepsilon$ if $\tau_\varepsilon^+=s$ due to the lack of upward jumps.
	The last term on the right hand side can be written, for any $p\geq 0$, as,
	\begin{equation*}
	\begin{split}
	\int_0^r  \mathbb E_\varepsilon & \left[ \widetilde f (X_{r-s})  \right]\mathbb P_x(\tau_\varepsilon^+\in\mathrm d s) \\
	= &    \int_0^r  \mathrm e^{-p(r-s)} \mathbb E_\varepsilon \left[ \widetilde f (X_{r-s})  \right] \mathrm e^{p(r-s)} \mathbb P_x(\tau_\varepsilon^+\in\mathrm d s) \\
	= & \int_0^r  \left( \widetilde f(\varepsilon) + \int_0^{r-s} \mathrm e^{-p u} \mathbb E_\varepsilon \left[ (\mathcal A-p)\widetilde f (X_u)  \right] \mathrm d u \right)  \mathrm e^{p(r-s)} \mathbb P_x(\tau_\varepsilon^+\in\mathrm d s) \\
	= &   \widetilde f(\varepsilon)  \int_0^r \mathrm e^{p(r-s)} \mathbb P_x(\tau_\varepsilon^+\in\mathrm d s) \\
	& +   \left. \int_0^{r-s} \mathrm e^{-p u} \mathbb E_\varepsilon \left[ (\mathcal A-p)\widetilde f (X_u)  \right] \mathrm d u   \int_0^{s} \mathrm e^{p(r-v)} \mathbb P_x(\tau_\varepsilon^+\in\mathrm d v) \right|_{s=0}^r \\
	&  + \int_0^r    \mathrm e^{-p (r-s)} \mathbb E_\varepsilon \left[ (\mathcal A-p)\widetilde f (X_{r-s})  \right]     \int_0^s \mathrm e^{p(r-v)} \mathbb P_x(\tau_\varepsilon^+\in\mathrm d v) \mathrm d s\\
	= &    \widetilde f(\varepsilon)\Lambda^{(p)}(x-\varepsilon,r)
	+ \int_0^r     \mathbb E_\varepsilon \left[ (\mathcal A-p)\widetilde f (X_{r-s})  \right]  \Lambda^{(p)}(x-\varepsilon,s) \mathrm d s,
	\end{split}
	\end{equation*}
	where the second equality follows from condition (i) in Theorem \ref{thm_mainresult} and the third and fourth equality follows by an integration by parts and Lemma \ref{lem_overshoot_Lamq} respectively.
	Combining the two computations completes the proof of the first identity.
	
	\medskip
	
	Next we prove the second identity. Let $x\leq b$. Using the first identity for $p=q$ and condition (ii) of Theorem \ref{thm_mainresult} together with Fubini, we get
\begin{equation}\label{split1}
\begin{split}
\mathbb{E}_x & \left[ \mathrm{e}^{-q\tau_0^-} \mathbf{1}_{\{\tau_0^-<\tau_b^+\}}
\mathbb{E}_{X_{\tau_0^-}} \left[ \widetilde{f}(X_r) \mathbf 1_{\{r<\tau_\varepsilon^+\}} \right] \right] \\
= & \mathbb{E}_x \left[ \mathrm{e}^{-q\tau_0^-} \mathbf{1}_{\{\tau_0^-<\tau_b^+\}}
\mathbb{E}_{X_{\tau_0^-}} \left[ \widetilde f (X_r)  \right]  \right]
-   \widetilde f(\varepsilon)\mathbb{E}_x \left[ \mathrm{e}^{-q\tau_0^-} \mathbf{1}_{\{\tau_0^-<\tau_b^+\}}  \Lambda^{(q)}(X_{\tau_0^-}-\varepsilon,r) \right] \\
 & - \int_0^r   \mathbb E_\varepsilon \left[ (\mathcal A-q)\widetilde f (X_{r-s})  \right] \mathbb{E}_x \left[ \mathrm{e}^{-q\tau_0^-} \mathbf{1}_{\{\tau_0^-<\tau_b^+\}}
\Lambda^{(q)}(X_{\tau_0^-}-\varepsilon,s) \right]\mathrm d s.
\end{split}
\end{equation}
For the first term on the right hand side of \eqref{split1}, we use Corollary 3 of \cite{ronnie}, which allows us to conclude that for $\widetilde h(x)=\mathbb E_x \left[ \widetilde f (X_r)  \right]$,
\begin{equation}\label{overshoot_identity}
\begin{split}
\mathbb E_x \Big[ \mathrm e^{-q \tau_0^-}  \mathbf 1_{\{ \tau_0^-<\tau_b^+ \}}   \widetilde h(X_{\tau_0^-})\Big]
= & \widetilde h(x)  -   \int_0^x   (\mathcal A-q) \widetilde h(z)   W^{(q)}(x- z) \mathrm dz \\
& -  \frac{W^{(q)}(x)}{W^{(q)}(b)} \left[ \widetilde h(b)  - \int_0^b   (\mathcal A-q) \widetilde h(z)   W^{(q)}(b- z)  \mathrm dz \right].
\end{split}
\end{equation}
Note that we are allowed to use Corollary 3 of \cite{ronnie} because of conditions (iii) and (iv) in Theorem \ref{thm_mainresult}. The proof is finished by combining \eqref{split1} with \eqref{overshoot_identity} and Lemma \ref{lem_overshoot_Lamq}.
\end{proof}

We now turn to the proof of Theorem \ref{thm_mainresult}. The steps below are similar to the ones in Section 3 of \cite{Loeffen}. In order to deal with the case where the sample paths of $X$ are of unbounded variation, we introduce the stopping time
 \begin{equation*}
 \tau_r^\varepsilon=\inf\{ t>r: t-g^\epsilon_t>r, X_{t-r}<0  \}, \quad \text{where $g^\varepsilon_t=\sup\{0\leq s\leq t:X_s\geq\varepsilon\}$}
 \end{equation*}
 for $\varepsilon\geq 0$. The stopping time $\tau_r^\epsilon$ is the first time that an excursion starting when $X$ gets below zero, ending before $X$ gets back up to $\epsilon$ and of length greater than $r$, has occurred. Note that $\tau_r=\tau_r^0$.
Let $0\leq \varepsilon <b$. We observe that for $x<0$, by the strong Markov property,
\begin{equation}
\label{distr_belowzero}
\begin{split}
\mathbb{E}_x \left[ \mathrm{e}^{-q\tau_r^\varepsilon} \widetilde f(X_{\tau_r^\varepsilon})\mathbf{1}_{\{\tau_r^\varepsilon <\tau_b^+\}} \right] = &
\mathbb{E}_x \left[ \mathrm{e}^{-q\tau_r^\varepsilon} \widetilde  f(X_{\tau_r^\varepsilon})\mathbf{1}_{\{\tau_r^\varepsilon<\tau_\varepsilon^+\}}  \right] \\
& + \mathbb{E}_x \left[ \mathrm{e}^{-q\tau_r^\varepsilon} \widetilde f(X_{\tau_r^\varepsilon})\mathbf{1}_{\{\tau_r^\varepsilon<\tau_b^+ ,\tau_r^\varepsilon>\tau_\varepsilon^+\}} \right] \\
= &  \mathrm{e}^{-q r} \mathbb{E}_x \left[  \widetilde f(X_{r})\mathbf{1}_{\{r<\tau_\varepsilon^+\}}  \right]\\
&+ \mathbb{E}_x \left[ \mathrm{e}^{-q\tau_\varepsilon^+} \mathbf{1}_{\{\tau_\varepsilon^+\leq r \}} \right] \mathbb{E}_\varepsilon \left[ \mathrm{e}^{-q\tau_r^\varepsilon} \widetilde f(X_{\tau_r^\varepsilon})\mathbf{1}_{\{\tau_r^\varepsilon<\tau_b^+\}} \right].
\end{split}
\end{equation}
Moreover, for $x\leq b$, we have by the strong Markov property and \eqref{distr_belowzero},
\begin{equation}
\label{distr_abovezero}
\begin{split}
\mathbb{E}_x \left[ \mathrm{e}^{-q\tau_r^\varepsilon} \widetilde f(X_{\tau_r^\varepsilon})\mathbf{1}_{\{\tau_r^\varepsilon<\tau_b^+\}} \right]
= & \mathbb{E}_x \left[ \mathbb{E}_x \left[ \left. \mathrm{e}^{-q\tau_r^\varepsilon}  \widetilde f(X_{\tau_r^\varepsilon})\mathbf{1}_{\{\tau_r^\varepsilon<\tau_b^+\}}  \right| \mathcal{F}_{\tau_0^-} \right] \right] \\
= & \mathbb{E}_x \left[ \mathrm{e}^{-q\tau_0^-} \mathbf{1}_{\{\tau_0^-<\tau_b^+\}}  \mathbb{E}_{X_{\tau_0^-}} \left[   \mathrm{e}^{-q\tau_r^\varepsilon} \widetilde f(X_{\tau_r^\varepsilon})\mathbf{1}_{\{\tau_r^\varepsilon<\tau_b^+\}} \right] \right] \\
= & \mathbb{E}_x \Big[ \mathrm{e}^{-q\tau_0^-} \mathbf{1}_{\{\tau_0^-<\tau_b^+\}}
\Big(    \mathrm{e}^{-q r} \mathbb{E}_{X_{\tau_0^-}} \left[ \widetilde f(X_{r})\mathbf{1}_{\{r<\tau_\varepsilon^+\}}  \right]  \\
& + \mathbb{E}_{X_{\tau_0^-}} \left[ \mathrm{e}^{-q\tau_\varepsilon^+} \mathbf{1}_{\{\tau_\varepsilon^+\leq r \}} \right] \mathbb{E}_\varepsilon  \left[ \mathrm{e}^{-q\tau_r^\varepsilon} \widetilde f(X_{\tau_r^\varepsilon})\mathbf{1}_{\{\tau_r^\varepsilon<\tau_b^+\}} \right] \Big)  \Big].
\end{split}
\end{equation}
Setting $x=\varepsilon$ in \eqref{distr_abovezero} and invoking Lemmas \ref{lem_overshoot_Lamq} and \ref{lem_key} yields,
\begin{equation*}
\begin{split}
\mathbb{E}_\varepsilon & \left[ \mathrm{e}^{-q\tau_r^\varepsilon} \widetilde f(X_{\tau_r^\varepsilon})\mathbf{1}_{\{\tau_r^\varepsilon<\tau_b^+\}} \right] \\
= &  \mathrm e^{-q r} \frac{\zeta_\varepsilon(\varepsilon;r,q) - \widetilde f(\varepsilon)  \Lambda^{(q)}(0,r) - \frac{W^{(q)}(\varepsilon)}{W^{(q)}(b)}  \left( \zeta_\varepsilon(b;r,q)- \widetilde f(\varepsilon)  \Lambda^{(q)}(b-\varepsilon,r) \right)  }{  1 - \mathrm{e}^{-qr} \left( \Lambda^{(q)}(0,r) - \frac{W^{(q)}(\varepsilon)}{W^{(q)}(b)}  \Lambda^{(q)}(b-\varepsilon,r) \right) } \\
= & \frac{ \zeta_\varepsilon(\varepsilon;r,q) - \widetilde f(\varepsilon)\mathrm e^{q r}  - \frac{W^{(q)}(\varepsilon)}{W^{(q)}(b)}    \zeta_\varepsilon(b;r,q)}{ \frac{W^{(q)}(\varepsilon)}{W^{(q)}(b)}  \Lambda^{(q)}(b-\varepsilon,r)} +\widetilde f(\varepsilon),
\end{split}
\end{equation*}
where we have used $\Lambda^{(q)}(0,r)=\mathrm e^{q r}$ (cf. Lemma \ref{lem_overshoot_Lamq}) in the second equality.
By definition of $\zeta_\varepsilon(\varepsilon;r,q)$ in Lemma \ref{lem_key}, the fact  $\Lambda^{(q)}(0,s) =\mathrm e^{q s}$ and a change of variables $u=r-s$, we can write
\begin{equation*}
\begin{split}
 \zeta_\varepsilon(\varepsilon;r,q) -\widetilde{f}(\varepsilon)\mathrm e^{q r}
= & \mathbb E_{\varepsilon} \left[ \widetilde f (X_r)  \right] -  \int_0^\varepsilon W^{(q)}(\varepsilon-z)    (\mathcal A-q) \mathbb E_{z} \left[ \widetilde f (X_r)  \right]  \mathrm d z \\
& -  \int_0^r  \mathbb E_\varepsilon \left[ (\mathcal A-q)\widetilde f (X_{u})  \right]  \mathrm e^{q (r-u)}  \mathrm{d}u
  - \widetilde{f}(\varepsilon)\mathrm e^{q r} \\
= & - \int_0^\varepsilon W^{(q)}(\varepsilon-z)    (\mathcal A-q) \mathbb E_{z} \left[ \widetilde f (X_r)  \right]  \mathrm d z,
\end{split}
\end{equation*}
where for the second equality we used condition (i) of Theorem \ref{thm_mainresult}.
Thus,
\begin{equation}\label{value_x=eps}
\begin{split}
\mathbb{E}_\varepsilon & \left[ \mathrm{e}^{-q\tau_r^\varepsilon} \widetilde f(X_{\tau_r^\varepsilon})\mathbf{1}_{\{\tau_r^\varepsilon<\tau_b^+\}} \right]\\& =     - \frac{ W^{(q)}(b)\int_0^\varepsilon \frac{ W^{(q)}(\varepsilon-z)}{W^{(q)}(\varepsilon) }(\mathcal A-q)\mathbb E_z\left[\widetilde f(X_r) \right]\mathrm d z +     \zeta_\varepsilon(b;r,q)}{  \Lambda^{(q)}(b-\varepsilon,r)  }
+   \widetilde f(\varepsilon).
\end{split}
\end{equation}
Next, we split the analysis up into two cases. First, assume that  $W^{(q)}(0)>0$. Then by setting $\varepsilon=0$ in \eqref{value_x=eps}, we deduce,
\begin{equation}\label{value_x=0}
\mathbb{E} \left[ \mathrm{e}^{-q\tau_r} \widetilde f(X_{\tau_r})\mathbf{1}_{\{\tau_r<\tau_b^+\}} \right] =
 - \frac{ \zeta(b;r,q)  }{  \Lambda^{(q)}(b,r)  } + \widetilde f(0).
\end{equation}
 Second, assume that $W^{(q)}(0)=0$. We now want to determine the limits as $\varepsilon\downarrow 0$ on both sides of \eqref{value_x=eps}. By condition (v) of Theorem \ref{thm_mainresult} and the well-known fact that, when $W^{(q)}(0)=0$, $W^{(q)}(\cdot)$ is continuously differentiable on $(0,\infty)$ with $W^{(q)\prime}(0)>0$, it follows that the integral in the numerator on the right hand side of \eqref{value_x=eps} converges to $0$ as $\varepsilon\downarrow 0$. By the monotone convergence theorem it follows that  $\Lambda^{(q)}(b-\varepsilon,r)\uparrow \Lambda^{(q)}(b,r)$ as $\epsilon\downarrow 0$. By conditions (ii) and (vi) of Theorem \ref{thm_mainresult}, one can show using the dominated convergence theorem that $\zeta_\varepsilon(b;r,q)\to\zeta(b;r,q)$ as $\varepsilon\downarrow 0$. Since by the assumptions $\widetilde f(x)$ is continuous at $x=0$ as well, we conclude that the right hand side of \eqref{value_x=eps} converges to the right hand side of \eqref{value_x=0}.
 To deal with the limit  as $\varepsilon\downarrow 0$ of the left hand side of \eqref{value_x=eps}, we introduce for $\epsilon>0$ the stopping time
 \begin{equation*}
 	\widetilde{\tau}^\epsilon_r=\inf\{ t>r: t-g_t>r, X_{t-r}<-\varepsilon  \}, \quad \text{where $g_t=\sup\{0\leq s\leq t:X_s\geq0\}$.}
 \end{equation*}
 We easily see that, as $\varepsilon\downarrow 0$, $\widetilde{\tau}^\varepsilon_r$ decreases monotonically to $\tau_r$ $\mathbb P$-a.s..
 Hence by spatial homogeneity, the assumed continuity of $\widetilde f$ and the fact that $X$ has c\`adl\`ag paths we have
 \begin{equation*}
 	\begin{split}
 		\lim_{\varepsilon\downarrow 0} \mathbb{E}_\varepsilon \left[ \mathrm{e}^{-q\tau_r^\varepsilon} \widetilde f(X_{\tau_r^\varepsilon})\mathbf{1}_{\{\tau_r^\varepsilon<\tau_b^+\}} \right]
 	 = & \lim_{\varepsilon\downarrow 0}\mathbb{E} \left[ \mathrm{e}^{-q\widetilde{\tau}_r^\varepsilon} \widetilde f(X_{\widetilde{\tau}_r^\varepsilon}+\varepsilon)\mathbf{1}_{\{ \widetilde{\tau}_r^\varepsilon<\tau_{b-\varepsilon}^+\}} \right] \\
 	 = &	\mathbb{E} \left[ \lim_{\varepsilon\downarrow 0} \mathrm{e}^{-q\widetilde{\tau}_r^\varepsilon} \widetilde f(X_{\widetilde{\tau}_r^\varepsilon}+\varepsilon)\mathbf{1}_{\{ \widetilde{\tau}_r^\varepsilon<\tau_{b-\varepsilon}^+\}} \right] \\
 	 = & \mathbb{E} \left[ \mathrm{e}^{-q \tau_r} \widetilde f(X_{\tau_r})\mathbf{1}_{\{\tau_r<\tau_b^+\}} \right],
 	\end{split}
 \end{equation*}
where for the second equality we used the dominated convergence theorem together with condition (vii) in Theorem \ref{thm_mainresult}. Hence \eqref{value_x=0} is proved in all cases. Now Theorem \ref{thm_mainresult} follows by setting $\varepsilon=0$ and using \eqref{value_x=0}, Lemmas \ref{lem_overshoot_Lamq} and \ref{lem_key} in \eqref{distr_abovezero}.

\subsection{Proof of Theorem \ref{theo:main2}}

The next theorem generalises Theorem \ref{theo:main2} by giving the $q$-potential measure of $X$ killed at $\tau_r\wedge\tau_b^+$.
\begin{theorem}\label{theo:general2}
Let $r,b>0$ and $q\geq 0$. Then for $x\leq b$ and $y\in\mathbb R$,
\begin{equation*}
\begin{split}
\int_0^\infty & \mathrm e^{-q (t-r)}  \mathbb P_x(X_t\in\mathrm d y, t<\tau_r\wedge\tau_b^+) \mathrm d t \\
 = &  \int_0^r \mathrm e^{q(r-s)}\mathbb P_x(X_s\in\mathrm d y) \mathrm{d}s  - \int_0^x W^{(q)}(x-z)\mathbb P_z(X_r\in\mathrm d y)\mathrm{d}z  \\
 & - \int_0^r \mathbb P(X_{r-s}\in\mathrm d y)
\Lambda^{(q)}(x,s)\mathrm{d}s   \\
& - \frac{\Lambda^{(q)}(x,r)}{\Lambda^{(q)}(b,r)} \Bigg( \int_0^r \mathrm e^{q(r-s)} \mathbb P_b(X_s\in\mathrm d y) \mathrm{d}s - \int_0^b W^{(q)}(b-z)\mathbb P_z(X_r\in\mathrm d y) \mathrm{d}z  \\
& -\int_0^r \mathbb P(X_{r-s}\in\mathrm d y) \Lambda^{(q)}(b,s)\mathrm{d}s \Bigg).
\end{split}
\end{equation*}
\end{theorem}

\paragraph{Proof of Theorem \ref{theo:main2}}
The proof follows from Theorem \ref{theo:general2} by taking Laplace transforms in $y$ on both sides and using Tonelli and \eqref{lapl_X}.
\exit

\subsection{Proof of Theorem \ref{theo:general2} and Proposition \ref{prop_resolpos}}

\paragraph{Proof of Theorem \ref{theo:general2}}
	Without loss of generality we assume that $q>0$ since the case $q=0$ can be dealt with by taking limits when $q\downarrow 0$ in combination with the monotone convergence theorem.
Let $\lambda\in[0,\Phi(q))$.
We start with the following simple observation based on the strong Markov property:
\begin{equation}\label{resol_step1}
\begin{split}
\mathbb{E}_x \left[ \int_0^{\tau_r\wedge\tau_b^+} \mathrm e^{-qt+\lambda X_t}  \mathrm{d}t \right]
= & \mathbb{E}_x \left[ \int_0^{\infty}  \mathrm e^{-qt+\lambda X_t}   \mathrm{d}t \right]
- \mathbb{E}_x \left[ \int_{\tau_r\wedge\tau_b^+}^\infty  \mathrm e^{-qt+ \lambda X_t} \mathrm{d}t \right] \\
= & \mathbb{E}_x \left[ \int_0^{\infty}  \mathrm e^{-qt+\lambda X_t}  \mathrm{d}t \right] \\
& - \mathbb{E}_x \left[ \mathrm e^{-q(\tau_r\wedge\tau_b^+)} \mathbb E_{X_{\tau_r\wedge\tau_b^+}}  \left[ \int_0^{\infty}  \mathrm e^{-qt+\lambda X_t} \mathrm{d}t \right] \right] \\
= & \frac1{q-\psi(\lambda)} \left( \mathrm e^{\lambda x}  - \mathbb{E}_x \left[ \mathrm e^{-q(\tau_r\wedge\tau_b^+) + \lambda X_{\tau_r\wedge\tau_b^+}}  \right] \right),
\end{split}
\end{equation}
where in the last equality we used \eqref{lapl_X} and the assumption that $\lambda\in[0,\Phi(q))$.
By Corollary \ref{LTpartwosided} and the fact that $X$ is spectrally negative,
\begin{equation}\label{resol_step2}
\begin{split}
\mathbb{E}_x \left[ \mathrm e^{-q(\tau_r\wedge\tau_b^+)+ \lambda X_{\tau_r\wedge\tau_b^+}} \right] = &
\mathbb{E}_x \left[ \mathrm{e}^{-q\tau_r+ \lambda X_{\tau_r}} \mathbf 1_{\{\tau_r<\tau_b^+\}} \right] +
 \mathrm e^{\lambda b} \mathbb E_x \left[ \mathrm e^{-q\tau_b^+} \mathbf 1_{\{\tau_r>\tau_b^+\}} \right] \\
= & \mathbb{E}_x \left[ \mathrm{e}^{-q\tau_r+ \lambda X_{\tau_r}} \mathbf{1}_{\{\tau_r<\tau_b^+\}} \right]  + \mathrm e^{\lambda b} \frac{\Lambda^{(q)}(x,r)}{\Lambda^{(q)}(b,r)}.
\end{split}
\end{equation}
Combining the two computations with Theorem \ref{thm_lapl}, Tonelli and \eqref{lapl_X} yields, for $\lambda\in[0,\Phi(q))$,
\begin{equation*}
\begin{split}
\int_{\mathbb R} \mathrm e^{\lambda y} & \int_0^\infty  \mathrm e^{-q (t-r)} \mathbb P_x(X_t\in\mathrm d y, t<\tau_r\wedge\tau_b^+) \mathrm d t \mathrm d y \\
= &   \frac{ \mathrm e^{\lambda x} \left( \mathrm e^{\psi(\lambda)r} - \mathrm e^{q r} \right)  }{ \psi(\lambda)-q} -     \int_0^x W^{(q)}(x-z)  \mathrm e^{\lambda z+\psi(\lambda)r}  \mathrm d z \\
& - \int_0^r   \mathrm e^{\psi(\lambda)(r-s) } \Lambda^{(q)}(x,s)   \mathrm{d}s  \\
& -  \frac{ \Lambda^{(q)}(x,r)   }{ \Lambda^{(q)}(b,r)   } \Bigg(  \frac{ \mathrm e^{\lambda b} \left( \mathrm e^{\psi(\lambda)r} - \mathrm e^{q r} \right) }{ \psi(\lambda)-q}   -      \int_0^b W^{(q)}(b-z)  \mathrm e^{\lambda z+\psi(\lambda)r}  \mathrm d z \\
& - \int_0^r   \mathrm e^{\psi(\lambda)(r-s) } \Lambda^{(q)}(b,s)   \mathrm{d}s \Bigg) \\
= &  \int_{\mathbb R} \mathrm e^{\lambda y} \Bigg\{ \int_0^r \mathrm e^{q(r-s)}\mathbb P_x(X_s\in\mathrm d y) \mathrm{d}s  - \int_0^x W^{(q)}(x-z)\mathbb P_z(X_r\in\mathrm d y)\mathrm{d}z  \\
& - \int_0^r \mathbb P(X_{r-s}\in\mathrm d y)
\Lambda^{(q)}(x,s)\mathrm{d}s   \\
& - \frac{\Lambda^{(q)}(x,r)}{\Lambda^{(q)}(b,r)} \Bigg( \int_0^r \mathrm e^{q(r-s)} \mathbb P_b(X_s\in\mathrm d y) \mathrm{d}s - \int_0^b W^{(q)}(b-z)\mathbb P_z(X_r\in\mathrm d y) \mathrm{d}z  \\
& -\int_0^r \mathbb P(X_{r-s}\in\mathrm d y) \Lambda^{(q)}(b,s)\mathrm{d}s \Bigg)\Bigg\}\mathrm d y.
\end{split}
\end{equation*}
Since the above holds for $\lambda$ in a non-empty interval (as $q>0$), Theorem \ref{theo:general2} follows by uniqueness of the Laplace transform.
\exit

\paragraph{Proof of Proposition \ref{prop_resolpos}}
As before, we assume $q>0$ without loss of generality.
Let $\lambda<\Phi(q)$. Then,
\begin{equation}\label{def_g}
 \begin{split}
g(x): = & \int_0^\infty \mathrm e^{\lambda y} \int_0^\infty \mathrm e^{-q t}\mathbb P_x(X_t\in\mathrm d y) \mathrm d t \\
= & \int_0^\infty \mathrm e^{\lambda y} \left( \frac{\mathrm e^{\Phi(q)(x-y)} }{\psi'(\Phi(q))} - W^{(q)}(x-y)   \right) \mathrm d y,
\end{split}
\end{equation}
 where the expression for the $q$-potential density of $X$ can be found in e.g. \cite[Cor. 8.9]{Kyprianou}. Since $W^{(q)}(x)=0$ for $x<0$, it follows that
 \begin{equation*}
 g(x)= \int_0^\infty \mathrm e^{\lambda y}  \frac{\mathrm e^{\Phi(q)(x-y)} }{\psi'(\Phi(q))}  \mathrm d y, \quad x\leq 0.
 \end{equation*}
 Therefore, noting that $X_{\tau_r}\leq 0$ on the event $\{\tau_r < \tau_b^+\}$, we have by an application of Theorem \ref{thm_lapl} with $\lambda=\Phi(q)$
 and since $\psi(\Phi(q))=q$,
\begin{equation*}
 \mathbb{E}_x \left[ \mathrm{e}^{-q\tau_r} g(X_{\tau_r}) \mathbf{1}_{\{\tau_r<\tau_b^+\}} \right]
 = \int_0^\infty \mathrm e^{\lambda y} \left(
  \frac{ \mathrm e^{\Phi(q)(x-y)} }{\psi'(\Phi(q))} - \frac{ \mathrm e^{\Phi(q)(b-y)} }{\psi'(\Phi(q))} \frac{\Lambda^{(q)}(x,r)}{\Lambda^{(q)}(b,r)} \right) \mathrm d y.
 \end{equation*}
 Using the same arguments that led to \eqref{resol_step1} and \eqref{resol_step2}, we get via \eqref{def_g},
 \begin{equation*}
 \begin{split}
\int_0^\infty \mathrm e^{\lambda y} \int_0^\infty \mathrm e^{-q t}  &  \mathbb P_x(X_t\in\mathrm d y,  t<\tau_r\wedge\tau_b^+) \mathrm d t \\
= & g(x) - g(b) \frac{\Lambda^{(q)}(x,r)}{\Lambda^{(q)}(b,r)} -  \mathbb{E}_x \left[ \mathrm{e}^{-q\tau_r} g(X_{\tau_r}) \mathbf{1}_{\{\tau_r<\tau_b^+\}} \right] \\
= & \int_0^\infty \mathrm e^{\lambda y} \left( \frac{\Lambda^{(q)}(x,r)}{\Lambda^{(q)}(b,r)} W^{(q)}(b-y) - W^{(q)}(x-y)  \right) \mathrm d y.
 \end{split}
 \end{equation*}
Proposition \ref{prop_resolpos} now follows by uniqueness of the Laplace transform.
\exit

\subsection{Proof of Corollaries \ref{cor_lapl_limb} and \ref{cor_laplresol_limb}}

 \paragraph{Proof of Corollary \ref{cor_lapl_limb}}
	Without loss of generality we assume that $q>0$ since the case $q=0$ can be dealt with by taking limits when $q\downarrow 0$ in combination with the monotone convergence theorem. The corollary follows from Theorem \ref{thm_lapl} by taking limits as $b\to\infty$ together with the monotone convergence theorem. Though we still need to identity the limit, as $b\to\infty$, of the right hand side of the identity in Theorem \ref{thm_lapl}. To this end, we note that by  \eqref{eq:scale2} and the dominated convergence theorem,
	\begin{equation}\label{lim_scaleLam}
	\begin{split}
	\lim_{b\to\infty} \frac{\Lambda^{(q)}(b,r)}{W^{(q)}(x)} = & \lim_{b\to\infty}  \int_0^\infty \mathrm e^{\Phi(q)z}\frac{W_{\Phi(q)}(b+z)}{ W_{\Phi(q)}(b)} \frac zr \mathbb P(X_r\in\mathrm d z) \\
	= & \int_0^\infty \mathrm e^{\Phi(q)z} \frac zr\mathbb P(X_r\in\mathrm d z).
	\end{split}
	\end{equation}
	Here the dominated convergence theorem is applicable because the assumption $q>0$ ensures that $\lim_{x\to\infty}W_{\Phi(q)}(x)<\infty$. The proof is then complete once we show that
	\begin{equation}\label{limit_ident}
	\lim_{b\to\infty}\frac{ \mathrm e^{\lambda b }  -  (\psi(\lambda)-q)  \int_0^b W^{(q)}(b-z)  \mathrm e^{\lambda z}  \mathrm d z }{ W^{(q)}(b) } = \frac{\psi(\lambda)-q }{ \lambda-\Phi(q)},
	\end{equation}
	where for $\lambda=\Phi(q)$ the right hand side is understood to be equal to $\lim_{\lambda\to\Phi(q)}\frac{\psi(\lambda)-q}{\lambda-\Phi(q)}=\psi'(\Phi(q))$.
	For this we split the analysis into three cases. First, assume $0\leq \lambda<\Phi(q)$. Then using \eqref{eq:scale2}, we get that the left hand side of \eqref{limit_ident} is given by,
	\begin{multline*}
	\lim_{b\to\infty}\frac{ \mathrm e^{\lambda b } }{ \mathrm e^{\Phi(q)b}W_{\Phi(q)}(b) } -  (\psi(\lambda)-q) \int_0^b  \frac{ \mathrm e^{\Phi(q)(b-z)}W_{\Phi(q)}(b-z) }{ \mathrm e^{\Phi(q)b}W_{\Phi(q)}(b) }   \mathrm e^{\lambda z}\mathrm d z \\
	= 0 -  (\psi(\lambda)-q)\int_0^\infty \mathrm e^{(\lambda-\Phi(q))z}\mathrm d z,
	\end{multline*}
	which equals the right hand side of \eqref{limit_ident}. Second, assume $\lambda>\Phi(q)$. Then
	\begin{equation*}
	\begin{split}
	\mathrm e^{\lambda b }  -  (\psi(\lambda)-q)  \int_0^b W^{(q)}(b-z)  \mathrm e^{\lambda z}  \mathrm d z
	= &	\mathrm e^{\lambda b } \left( 1 - (\psi(\lambda)-q)\int_0^b  \mathrm e^{-\lambda z} W^{(q)}(z)   \mathrm d z \right) \\
	= &	\mathrm e^{\lambda b } \int_b^\infty  \mathrm e^{-\lambda z} W^{(q)}(z)   \mathrm d z \\
	= & \int_0^\infty  \mathrm e^{-\lambda y} W^{(q)}(y+b)   \mathrm d y,
	\end{split}
	\end{equation*}
	where we used \eqref{eq:scale} for the second equality. Now \eqref{limit_ident} follows for the case $\lambda>\Phi(q)$ via \eqref{eq:scale2} and the dominated convergence theorem. Finally, when $\lambda=\Phi(q)$,   \eqref{limit_ident} follows easily by  \eqref{eq:scale2}.
 \exit

\paragraph{Proof of Corollary \ref{cor_laplresol_limb}}
As previously, we assume without loss of generality that $q>0$. The corollary follows from Theorem \ref{theo:main2} by taking limits as $b\to\infty$ together with the monotone convergence theorem. What is left is to determine the limit of the right hand side of the identity in Theorem \ref{theo:main2}.
By \eqref{lim_scaleLam} and \eqref{eq:scale2}, we have
\begin{equation*}
	\lim_{b\to\infty} \frac{\mathrm e^{\lambda b}}{\Lambda^{(q)}(b,r)} =
	\begin{cases}
	\infty & \text{if $\lambda>\Phi(q)$}, \\
	0 & \text{if $0\leq \lambda<\Phi(q)$}.
	\end{cases}
\end{equation*}	
Using this observation in combination with \eqref{lim_scaleLam} and \eqref{limit_ident}, we deduce the corollary in the case where $\lambda\neq \Phi(q)$.
Now assume  $\lambda= \Phi(q)$. Note that then the ratio $\frac{ 1 - \mathrm e^{-(\psi(\lambda)-q) r}  }{ \psi(\lambda)-q}$ is equal to $r$ as it is understood in the limiting sense as $\lambda\to\Phi(q)$ when $\lambda=\Phi(q)$. We have
\begin{equation*}
\begin{split}
\lim_{b\to\infty} \frac{ \mathrm e^{\Phi(q) b} r -   \int_0^b W^{(q)}(b-z)  \mathrm e^{\Phi(q) z}  \mathrm d z }{W^{(q)}(b)}
= \lim_{b\to\infty} \left( \frac{r}{W_{\Phi(q)}(b)} - \int_0^b \frac{ W_{\Phi(q)}(b-z) }{ W_{\Phi(q)}(b) }\mathrm d z \right)
\end{split}
\end{equation*}
and this equals $-\infty$ since given $q>0$, $W_{\Phi(q)}$ is a strictly increasing function with $\lim_{x\to\infty} W_{\Phi(q)}(x)\in(0,\infty)$. In combination with \eqref{lim_scaleLam}, this shows that the left hand side of the identity in Corollary \ref{cor_laplresol_limb} equals $+\infty$ in the case where $\lambda=\Phi(q)$.
\exit

\section*{Acknowledgement}
Part of the work was completed while Z. Palmowski was visiting the University of Manchester. R. Loeffen and Z. Palmowski are grateful to the LMS for supporting this visit by an LMS visitors grant (reference: 21610).
B.A. Surya acknowledges the support and hospitality provided by the
Center for Applied Probability of Columbia University in the City of
New York during his academic visit in May 2014 in which he started
working on the problem.
He also appreciates inputs from participants of the 2016 Wellington Workshop on Probability and Mathematical Statistics as well as ANU College of Business and Economics Seminar Series at parts of this research were presented.
This research is financially supported by Victoria University PBRF Research Grants \# 212885 and \# 214168 for which Budhi Surya would like to acknowledge.
This work is partially supported by the National Science Centre under the grant
 2016/23/B/HS4/00566 (2017-2020).


\begin{thebibliography}{21}

\bibitem{AIZ}
Albrecher, H., Ivanovs, J. and Zhou, X. (2016). Exit identities for L\'vy
processes observed at Poisson arrival times. \textit{Bernoulli} \textbf{22(3)}, 1364--1382.


\bibitem{Baurdoux} Baurdoux, E.J., Pardo, J.C., Perez, J.L. and Renaud, J.-F. (2016).
Gerber-Shiu functionals at Parisian ruin for L\'evy insurance risk
processes. \textit{Journal of Applied Probability}  \textbf{53(2)}, 572--584.

\bibitem{Bertoin} Bertoin, J. (1996). \textit{L\'evy Processes. Cambridge Tracts in
Mathematics} \textbf{121}. Cambridge: Cambridge Univ. Press.


\bibitem{Broadie} Broadie, M., Chernov, M. and Sundaresan, S.
(2007). Optimal debt and equity values in the presence of Chapter 7
and Chapter 11. \textit{J. Finance} \textbf{LXII(3)}, 1341-1377.


\bibitem{Chesney} Chesney, M., Jeanblanc-Picqu\'e, M. and Yor, M. (1997). Brownian excursions
and Parisian barrier options. \textit{Adv. Appl. Probab.} \textbf{29}, 165-184.


\bibitem{Czarna} Czarna, I. and Palmowski, Z. (2011). Ruin
probability with Parisian delay for a spectrally negative L\'evy
process. \textit{J. Appl. Probab.} \textbf{48}, 984-1002.



\bibitem{Dassios2009a} Dassios, A. and Wu, S. (2009). Semi-Markov model for
excursions and occupation time of Markov processes. Working paper,
LSE London. Available at http://stats.lse.ac.uk/angelos.

\bibitem{Dassios2009b} Dassios, A. and Wu, S. (2009). Parisian ruin with
exponential claims. Working paper, LSE London. Available at
http://stats.lse.ac.uk/angelos.


\bibitem{Francois} Francois, P. and Morellec, E. (2004).
Capital structure and asset prices: Some effects of bankruptcy
procedures. \textit{J. Business} \textbf{77}, 387-411.

\bibitem{Kuznetzov} Kusnetzov, A., Kyprianou, A.E. and Rivero, V.
(2013). \textit{The Theory of Scale Functions for Spectrally
Negative L\'evy Processes}, L\'evy Matters II, Springer Lecture
Notes in Mathematics.


\bibitem{Kyprianou} Kyprianou, A.E. (2014). \textit{Introductory
Lectures on Fluctuations of L\'evy Processes with Applications.
Universitext.} Berlin: Springer.


\bibitem{Landriault} Landriault, D., Renaud, J-F. and Zhou, X. (2014).
An insurance risk model with Parisian implementation delays.
\textit{Methodol. Comput. Appl. Probab} \textbf{16(3)}, 583-607.

\bibitem{CzarnaRenaud}
Lkabous, M.A., Czarna, I. and Renaud, J-F. (2016).
Parisian ruin for a refracted L\'evy process.
\textit{Insurance: Mathematics and Economics}, see https://arxiv.org/abs/1603.09324.

\bibitem{Loeffen} Loeffen, R., Czarna, I. and Palmowski, Z. (2013). Parisian ruin probability for spectrally
negative L\'evy processes \textit{Bernoulli}, \textbf{19(2)}, 599-609.

\bibitem{ronnie}
Loeffen, R. (2014).
On obtaining simple identities for overshoots of spectrally negative L\'evy processes.
Submitted for publication, see http://arxiv.org/abs/1410.5341.


\bibitem{Sato}
Sato, K. (1999). {\it L\'evy processes
and infinitely divisible distributions}. Cambridge University Press.

\bibitem{Surya} Surya, B.A. (2008). Evaluating scale function of spectrally
negative L\'evy processes. \textit{J. Appl. Probab.}
\textbf{45}, 135-149.



\end{thebibliography}
\end{document}